\documentclass[10pt]{amsart}
\usepackage{enumerate,amsmath,amssymb,latexsym,
amsfonts, amsthm, amscd, MnSymbol}

\theoremstyle{plain}

\newtheorem{theorem}{Theorem}

\numberwithin{equation}{section}

\newcommand{\R}{\mathbb{R}}
\newcommand{\C}{\mathbb{C}}

\newcommand{\ii}{\rm i}

\begin{document}

\title {Higher Trigonometry: A Class Of Nonlinear Systems}

\date{}

\author[P.L. Robinson]{P.L. Robinson}

\address{Department of Mathematics \\ University of Florida \\ Gainesville FL 32611  USA }

\email[]{paulr@ufl.edu}

 \keywords{}

\begin{abstract}
We study the initial value problem `$s\,' = c^{p - 1}, \; c\,' = -s^{p - 1}; \; \; s(0) = 0, \; c(0) = 1$' (both as a real system and as a complex system) for each integer $p > 2$, considering separately the cases `$p$ even' and `$p$ odd'.  
\end{abstract}

\maketitle

\medbreak

\section*{Introduction} 

\medbreak 

When viewed in terms of differential equations, trigonometry may be said to derive from the linear first-order system 
$$s\,' = c, \; c\,' = -s; \; \; s(0) = 0, \; c(0) = 1.$$
This system is the (exceptional) root of a sequence, whose subsequent members are the nonlinear first-order systems 
$$s\,' = c^{p - 1}, \; c\,' = -s^{p - 1}; \; \; s(0) = 0, \; c(0) = 1$$
for integers $p$ greater than two. Some of these nonlinear systems have received individual attention in the literature. Most extensively studied has been the case $p = 3$: the solutions $s$ and $c$ in this case are the Dixonian elliptic functions, named in honour of A.C. Dixon [2]; see also [1], [4] and [6]. The case $p = 4$ was considered in [7] and earlier, independently and from a different angle, in [5]; in that case, the quadratic expressions $s c, s^2$ and $c^2$ are elliptic functions though $s$ and $c$ themselves are not. More recently, the case $p = 6$ was considered in [8]; among other things, it was there shown that the quartic expression $s^2 c^2$ is elliptic but that $s^4, s^3 c, s c^3, c^4$ are not. 

\medbreak 

We here place on record certain properties that are shared by the nonlinear systems in this sequence, treating them first as real systems and then as complex systems. The behaviour of these systems is sharply dependent on the parity of $p$. In the real case, if $p$ is even then $s$ and $c$ are bounded, whereas if $p$ is odd then $s$ and $c$ undergo finite-time blow-up. In the complex case, if $p$ is even then $s$ and $c$ are holomorphic with real period in a band centred on the real axis, whereas if $p$ is odd then $s$ and $c$ have no such property. 

\medbreak 

\section*{Real Systems} 

\medbreak 

Here, we consider 
$$s\,' = c^{p - 1}, \; c\,' = -s^{p - 1}; \; \; s(0) = 0, \; c(0) = 1$$
as a real initial value problem ({\bf IVP}$\R$): that is, we consider real-valued solutions $s$ and $c$ defined on intervals containing $0$. As noted in the Introduction, the parity of the integer $p > 2$ exerts considerable influence; accordingly, we consider the cases `$p$ even' and `$p$ odd' separately. 

\medbreak 

{\it Let the integer $p > 2$ be even}. 

\medbreak 

We begin by noting that the even function $\xi \mapsto (1 - \xi^p)^{- 1 + 1/p}$ is integrable over the open interval $(-1, 1)$ and that 
$$A_p : = \int_0^1 \frac{{\rm d} \xi}{(1 - \xi^p)^{1 - 1/p}} = \frac{1}{p} \, \frac{\Gamma(\tfrac{1}{p})^2}{\Gamma( \tfrac{2}{p})}\, .$$

\medbreak 

Now, let us define 
$$\sigma_0 : [-1, 1] \to [- A_p, A_p]$$ 
by the rule that if $-1 \leqslant x \leqslant 1$ then 
$$\sigma_0 (x) = \int_0^x  \frac{{\rm d} \xi}{(1 - \xi^p)^{1 - 1/p}}\,.$$ 
This odd function $\sigma_0$ is continuous on $[-1, 1]$ and differentiable on $(-1, 1)$: if $-1 < x < 1$ then 
$$\sigma_0'(x) = (1 - x^p)^{- 1 + 1/p} > 0\,;$$
in particular, $\sigma_0'(x) \to \infty$ as $x \to \pm 1$. Consequently, $\sigma_0$ has an odd inverse 
$$s_0 : [- A_p, A_p] \to [-1, 1]$$ 
such that  
$$s_0\,' = (1 - s_0^p)^{1 - 1/p}$$ 
with vanishing one-sided derivatives at the endpoints and with $s_0(\pm A_p) = \pm 1$. In terms of $s_0$ we define the continuous even function 
$$c_0: [- A_p, A_p] \to [0, 1]$$
by 
$$c_0 = (1 - s_0^p)^{1/p} \,.$$ 
Within the open interval $(- A_p, A_p)$ we have $s_0\,' = c_0^{p - 1}$ of course, along with 
$$c_0\,' = (1/p) (1 - s_0^p)^{-1 + 1/p} \, (- p s_0^{p - 1} s_0\,') = - s_0^{p - 1}$$
by the chain rule; at the endpoints, $c_0$ vanishes and $c_0\,'(\pm A_p) = \mp 1$ as one-sided derivatives. 

\medbreak 

\begin{theorem} \label{-AA}
If $p > 2$ is even then {\bf IVP$\R$} has solution pair $(s_0, c_0)$ on the interval $[- A_p, A_p ]$ where $s_0$ is the inverse of 
$$\sigma_0 : [-1, 1] \to [ -A_p, A_p] : x \mapsto \int_0^x  \frac{{\rm d} \xi}{(1 - \xi^p)^{1 - 1/p}}$$ 
and where 
$$c_0 = (1 - s_0^p)^{1/p}.$$
\end{theorem} 

\begin{proof} 
That $(s_0, c_0)$ is defined and satisfies the differential equations $s_0\,' = c_0^{p - 1}$ and $c_0\,' = - s_0^{p - 1}$ was established prior to the Theorem; that the initial conditions $s_0(0) = 0$ and $c_0(0) = 1$ are satisfied is plain. 
\end{proof} 

\medbreak 

This solution pair $(s_0, c_0)$ extends beyond the interval $[- A_p, A_p ]$: in fact, it extends to the whole real line, with $4 A_p$ as period. In order to see that this is so, we extend initially to the interval $[- 2 A_p, 2 A_p ]$. Whereas on $[- A_p, A_p ]$ we took $s$ to be fundamental, it is convenient to let $c$ carry the burden of this extension. 

\medbreak 

Explicitly, for $- 1 \leqslant x \leqslant 1$ let us write  
$$\gamma_+ (x) = \int_x^1 \frac{{\rm d} \xi}{(1 - \xi^p)^{1 - 1/p}}$$
and thereby define a continuous function 
$$\gamma_+ : [-1, 1] \to [0, 2 A_p].$$
If $-1 < x < 1$ then 
$$\gamma_+'(x) = - \, (1 - x^p)^{- 1 + 1/p} < 0$$
so that $\gamma_+$ is strictly decreasing, with $\gamma_+ (-1) = 2 A_p$ and $\gamma_+ (1) = 0$. It follows that $\gamma_+$ has a continuous  inverse function 
$$c_+ : [0, 2 A_p] \to [-1, 1]$$ 
with 
$$c_+ (0) = 1, \; c_+(A_p) = 0, \; c_+(2 A_p) = -1;$$
its derivative is given by 
$$c_+' = - (1 - c^p)^{1 - 1/p}$$
in the open interval, and vanishes (as a one-sided derivative) at the endpoints. 

\medbreak 

\begin{theorem} \label{+}
If $p > 2$ is even then {\bf IVP$\R$} has solution pair $(s_+, c_+)$ on the interval $[0, 2 A_p]$ where $c_+$ is the inverse of 
$$\gamma_+ : [-1, 1] \to [0, 2 A_p]: x \mapsto \int_x^1 \frac{{\rm d} \xi}{(1 - \xi^p)^{1 - 1/p}}$$ 
and where 
$$s_+ = (1 - c_+^p)^{1/p}.$$ 
\end{theorem} 

\begin{proof} 
An easy exercise parallel to the proof of Theorem \ref{-AA}. 
\end{proof} 

\medbreak 

Likewise, {\bf IVP}$\R$ has solution pair $(s_-, c_-)$ on the interval $[- 2 A_p, 0]$ where $c_-$ is the inverse of 
$$\gamma_- : [-1, 1] \to [-2 A_p, 0] : x \mapsto - \int_x^1 \frac{{\rm d} \xi}{(1 - \xi^p)^{1 - 1/p}}$$ 
and where 
$$s_- = - \, (1 - c_-^p)^{1/p}.$$

\medbreak 

We can piece together the solutions on $[- 2 A_p, 0]$ and $[0, 2 A_p]$ to obtain a solution on the interval $[ - 2 A_p, 2 A_p].$ 

\medbreak 

\begin{theorem} \label{-2A2A}
If $p > 2$ is even then {\bf IVP$\R$} has solution pair $(s, c)$ on the interval $[ - 2 A_p, 2 A_p]$ where 
$$s|_{[- 2 A_p, 0]} = s_-, \; c|_{[- 2 A_p, 0]} = c_-$$
and 
$$s|_{[0, 2 A_p]} = s_+, \; c|_{[0, 2 A_p]} = c_+.$$
\end{theorem} 

\begin{proof} 
The functions $(s_-, c_-)$ and $(s_+, c_+)$ agree at the origin, as do their one-sided derivatives. 
\end{proof} 

\medbreak 

We remark here that $c$ is even and $s$ is odd. To see that $c$ is even, note that if $-1 \leqslant x \leqslant 1$ then 
$$\gamma_+ (x) = \Big( \int_0^1 - \int_0^x \Big) \: \frac{{\rm d} \xi}{(1 - \xi^p)^{1 - 1/p}} = A_p - \sigma_0 (x)$$ 
while 
$$\gamma_- (x) = - \: \Big( \int_0^1 - \int_0^x \Big) \: \frac{{\rm d} \xi}{(1 - \xi^p)^{1 - 1/p}} = - A_p + \sigma_0 (x).$$ 
Now, if $0 \leqslant u \leqslant 2 A_p$ then $A_p - u \in [- A_p, A_p]$ so that $A_p - u = \sigma_0 (x)$ for a unique $x \in [-1, 1]$ and the formulae above show that $c_+ (u) = x = c_- (- u)$. To see that $s$ is odd, use this result along with $s_+ = (1 - c_+^p)^{1/p}$ and $s_- = - \, (1 - c_-^p)^{1/p}$. 

\medbreak 

We remark further that the pair $(s, c)$ of Theorem \ref{-2A2A} restricts to the interval $[- A_p, A_p]$ as the pair $(s_0, c_0)$ of Theorem \ref{-AA}: this is clear on account of the classical Picard existence-uniqueness theorem; we leave as an exercise its verification from the very definitions of the pairs involved. 

\medbreak 

As claimed, the solution pair $(s, c)$ extends naturally to the whole real line. 

\medbreak 

\begin{theorem} \label{realline}
If $p > 2$ is even then the solution $(s, c)$ to {\bf IVP}$\R$ extends to the whole real line with $4 A_p$ as period. 
\end{theorem} 

\begin{proof} 
Note that $c(- 2 A_p) = - 1 = c(2 A_p)$ and $s(- 2 A_p) = 0 = s(2 A_p)$; note also that (as one-sided derivatives) $c\,' (-2 A_p) = 0 = c\,'(2 A_p)$ and $s\,'(- 2 A_p) = -1 = s\,'(2 A_p).$ These agreements ensure that $4 A_p$-periodic extension of $(s, c)$ yields a solution pair to {\bf IVP}$\R$ with domain the whole of $\R$. 
\end{proof} 

\medbreak 

Of course, the periodically-extended solution $(s, c)$ to {\bf IVP}$\R$ maintains the parity of the original solution, in that $s$ is odd and $c$ is even. 

\medbreak 

{\it Let the integer $p > 2$ be odd}. 

\medbreak 

We begin by noting that the function $\eta \mapsto (1 + \eta^p)^{- 1 + 1/p}$ is integrable over $(0, \infty)$ and that 
$$B_p : = \int_0^{\infty} \frac{{\rm d} \eta}{(1 + \eta^p)^{1 - 1/p}} = \frac{1}{p} \frac{\Gamma(1 - \tfrac{2}{p}) \, \Gamma(\tfrac{1}{p})}{\Gamma(1 - \tfrac{1}{p})} \, .$$

\medbreak 

Now, let us define 
$$\sigma : (- \, \infty, 1] \to (- \, B_p, A_p]$$
by the rule that if $- \, \infty < x \leqslant 1$ then 
$$\sigma (x) = \int_0^x  \frac{{\rm d} \xi}{(1 - \xi^p)^{1 - 1/p}}\,;$$
here, note that if $\xi < 0$ then $\xi^p < 0$ because $p$ is odd, so $\sigma$ is properly defined. The derivative $\sigma'$ is strictly positive, so that $\sigma$ is strictly increasing; further, $\sigma(1) = A_p$ and if $x \downarrow - \infty$ then $- x \uparrow \infty$ so 
$$\sigma (x) = \int_0^x \frac{{\rm d} \xi}{(1 - \xi^p)^{1 - 1/p}} = - \, \int_0^{- x} \frac{{\rm d} \eta}{(1 + \eta^p)^{1 - 1/p}} \, \downarrow - B_p \, .$$
Consequently, $\sigma$ has an inverse 
$$s: (- B_p, A_p] \to (- \infty, 1]$$
such that 
$$s\,' = (1 - s^p)^{1 - 1/p} .$$ 

\medbreak 

\begin{theorem} \label{odd-}
If $p > 2$ is odd then {\bf IVP}$\R$ has solution pair $(s, c)$ on the interval $(- B_p, A_p]$ where $s$ is the inverse of 
$$\sigma: (- \infty, 1] \to (- B_p, A_p] : x \mapsto \int_0^x  \frac{{\rm d} \xi}{(1 - \xi^p)^{1 - 1/p}}$$ 
and where 
$$c = (1 - s^p)^{1/p}.$$
\end{theorem} 

\begin{proof} 
Verification that $s: (- B_p, A_p] \to (- \infty, 1]$ and $c: (- B_p, A_p] \to [0, \infty)$ satisfy {\bf IVP}$\R$ proceeds essentially as did the similar verification that led up to Theorem \ref{-AA}. 
\end{proof} 

\medbreak 

Likewise, {\bf IVP}$\R$ has a solution pair $(s, c)$ on the interval $[0, A_p + B_p)$ where 
$$c : [0, A_p + B_p) \to (- \infty, 1]$$
is the inverse of 
$$\gamma : (- \infty, 1] \to [0, A_p + B_p) : x \mapsto \int_x^1 \frac{{\rm d} \xi}{(1 - \xi^p)^{1 - 1/p}}$$
and where 
$$s = (1 - c^p)^{1/p}.$$

\medbreak 

\begin{theorem} \label{odd}
If $p > 2$ is odd then {\bf IVP}$\R$ has a solution pair $(s, c)$ on the interval $(- B_p, A_p + B_p)$ as maximal domain. 
\end{theorem} 

\begin{proof} 
The solution pairs displayed in Theorem \ref{odd-} and the subsequent comment agree throughout the intersection $[0, A_p]$ of their domains (by virtue of the uniqueness clause in the classical Picard theorem or, as an exercise, directly from their definitions) and therefore patch together to yield a solution pair on the union $(-B_p, A_p + B_p)$ of their domains. This union is plainly the maximal domain: if $x \downarrow - B_p$ then $s(x) \downarrow - \infty$ and $c(x) \uparrow \infty$; if $x \uparrow A_p + B_p$ then $s(x) \uparrow \infty$ and $c(x) \downarrow - \infty$. 
\end{proof} 

\medbreak 

Here we see the sharp contrast between the cases `$p$ even' and `$p$ odd': if $p$ is even, then $s$ and $c$ are bounded and periodic on the whole real line; if $p$ is odd, then $s$ and $c$ suffer finite-time blow-up in each direction, having the same bounded open interval as their maximal domain. 

\medbreak

Regardless of the parity of $p$, it is straightforward to compare $A_p$ and $B_p$. We may calculate their ratio as follows, using the Euler reflexion formula for the Gamma function: 
$$\frac{A_p}{B_p} = \frac{\Gamma(1 - \tfrac{1}{p})\,\Gamma(\tfrac{1}{p})}{\Gamma(1 - \tfrac{2}{p})\,\Gamma(\tfrac{2}{p})} = \frac{\sin \tfrac{2 \pi}{p}}{\sin \tfrac{\pi}{p}} = 2 \cos \tfrac{\pi}{p}\, .$$
In particular, if $p > 3$ then $\pi/p < \pi/3$ so that $1/2 < \cos (\pi/p) < 1$ and therefore $B_p < A_p < 2 B_p$. The exceptionial `Dixonian' case has $A_3 = B_3$. 

\medbreak 

\section*{Complex Systems}

\medbreak 

We now pass on to a fresh consideration of 
$$s\,' = c^{p - 1}, \; c\,' = -s^{p - 1}; \; \; s(0) = 0, \; c(0) = 1$$
as a complex initial value problem ({\bf IVP}$\C$): that is, we consider complex-valued solutions $s$ and $c$ defined on complex domains containing $0$. Again, the parity of the integer $p$ has significant consequences for the behaviour of the system; however, we shall begin with some general observations that do not depend on parity. Throughout this section, it is to be understood that $p > 2$, though this is not stated explicitly in the enunciation of theorems. 

\medbreak 

First of all, we record the following counterpart to the trigonometric `Pythagorean' identity $\cos^2 + \sin^2 = 1$. 

\medbreak 

\begin{theorem} \label{Pyth}
Any solution pair to {\bf IVP}$\C$ on a connected open neighbourhood of $0$ satisfies the identity $s^p + c^p = 1$. 
\end{theorem} 

\begin{proof} 
Differentiate $s^p + c^p$; evaluate $s^p + c^p$ at $0$. 
\end{proof} 

\medbreak 

For {\bf IVP}$\C$ there exists a unique solution pair defined in a suitably small disc about $0$; this claim is justified by a simple application of the classical Picard existence-uniqueness theorem for initial value problems, as follows. 

\medbreak 

\begin{theorem} \label{r}
The system {\bf IVP}$\C$ has a unique solution pair $(s, c)$ in the open disc $B_r (0)$ about $0$ of radius $r = (p - 2)^{p - 2} / (p - 1)^{p - 1}.$  
\end{theorem} 

\begin{proof} 
Fix $b > 0$: if $|s| \leqslant b$ and $|c - 1| \leqslant b$ then $\max \, (|s^{p - 1}|, \, |c^{p - 1}| ) \leqslant (b + 1)^{p - 1}$; it follows from the Picard theorem (for which, see Section 2.3 of [3]) that {\bf IVP}$\C$ has a unique holomorphic solution pair in the open disc about $0$ of radius $b / (b + 1)^{p - 1}$. This radius is maximized to $(p - 2)^{p - 2} / (p - 1)^{p - 1}$ by taking $b \uparrow 1 / (p - 2)$. 
\end{proof} 

\medbreak 

We remark that the radius of the disc can be increased to $r^{1/p}$ by instead solving the initial value problem `$s\,' = (1 - s^p)^{1 - 1/p}; \, s(0) = 0$' for $s$ alone and then defining $c = (1 - s^p)^{1/p}$; here, principal-valued powers are taken and Theorem \ref{Pyth} is involved. 

\medbreak 

The solution pair $(s, c)$ is `real' in the following sense. 

\medbreak 

\begin{theorem} \label{conj}
If $z \in B_r(0)$ then $\overline{s(z)} = s(\overline{z})$ and $\overline{c(z)} = c(\overline{z}).$
\end{theorem} 

\begin{proof} 
For $z \in B_r (0)$ define $S(z) = \overline{s(\overline{z})}$ and $C(z) = \overline{c(\overline{z})}$. By direct calculation, the pair $(S, C)$ satisfies {\bf IVP}$\C$; by Theorem \ref{r}, $(S, C) = (s, c)$. 
\end{proof} 

\medbreak 

Let us write 
$$\alpha = e^{2 \pi \ii /p}$$
and write 
$$\beta = e^{ \pi \ii /p}$$
so that $\beta^2 = \alpha$.

\medbreak 

The solution pair $(s, c)$ of Theorem \ref{r} exhibits a $p$-fold rotational symmetry: under the action of $\alpha$ by multiplication, $s$ is equivariant and $c$ is invariant; so $s^p$ is also invariant. 

\medbreak 

\begin{theorem} \label{alpha}
If $z \in B_r (0)$ then $s(\alpha z) = \alpha s(z)$ and $c( \alpha z) = c(z)$. 
\end{theorem} 

\begin{proof} 
For $z \in B_r (0)$ define $S(z) = \overline{\alpha} s(\alpha z)$ and $C(z) = c(\alpha z)$. By direct calculation, the pair $(S, C)$ satisfies {\bf IVP}$\C$; by Theorem \ref{r}, $(S, C) = (s, c)$. 
\end{proof} 

\medbreak 

Thus far, we have discussed the solution pair $(s, c)$ only on the disc $B_r (0)$. We now wish to consider the question of extending the domain of this pair further into the complex plane. Note that the symmetries of the pair under conjugation and rotation will continue to hold for extensions to appropriately symmetric connected domains. 

\medbreak 

Regarding the possibility that an isolated singularity of an extended $s$ or $c$ might be a pole, we have the following result. 

\medbreak 

\begin{theorem} \label{pole}
An extension of $s$ or $c$ can have a pole only if $p = 3$. 
\end{theorem} 

\begin{proof} 
A pole of either extended function is a pole of the other. Consider a pole, of order $m$ for $s$ and $n$ for $c$. From $s\,' = c^{p - 1}$ follows $m + 1 = n (p - 1)$ and from $c\,' = - s^{p - 1}$ follows $n + 1 = m (p - 1)$. These equations in $m$ and $n$ have the unique solution $m = n = 1/(p - 2)$. As $m$ and $n$ are positive integers, $p = 3$ follows. 
\end{proof} 

\medbreak 

When $p = 3$, the functions $s$ and $c$ extend to simple-poled (Dixonian) elliptic functions in the plane: see [2] and [6]. When $p = 4$, the squares $s^2$ and $c^2$ extend to simply-poled elliptic functions in the plane: see [7]. 

\medbreak 

Extension of the functions $s$ and $c$ is intimately connected to extension of their quotient. Before we investigate this link, we study the quotient $s/c$. 
For obvious reasons, we denote this quotient by $t$; further, we continue this notation for such extensions as appear below. 

\medbreak 

\begin{theorem} \label{t}
The quotient $t = s/c$ satisfies the differential equation 
$$(t\,')^p = (1 + t^p)^2$$ 
on $B_r(0)$ and the initial conditions 
$$t\,'(0) = 1 \; \; and \; \; t(0) = 0.$$
\end{theorem} 

\begin{proof} 
Regarding the differential equation, 
$$(s/c)\,' = \frac{c s\,' - s c\,'}{c^2} = \frac{c^p + s^p}{c^2} = 1/c^2$$
so 
$$(s/c)\,'\,^p = \frac{1}{c^{2 p}} = \frac{(c^p + s^p)^2}{c^{2 p}} = \Big(1 + \big(s/c\big)^p\Big)^2$$
by Theorem \ref{Pyth}. Regarding the initial conditions, there is little to say. 
\end{proof} 

\medbreak 

At this point, we recall some classical Schwarz-Christoffel theory: the rule 
$$\tau: w \mapsto \int_0^w \frac{{\rm d} \zeta}{(1 + \zeta^p)^{2/p}}$$
maps the open unit disc $\mathbb{D}$ conformally to the open regular $p$-gon $\mathbb{P}$ that is centred at $0$ and has the positive real number 
$$K_p = \int_0^1 \frac{{\rm d} \zeta}{(1 + \zeta^p)^{2/p}}$$
as the midpoint of one of its edges; further, the same rule maps the closed unit disc $\overline{\mathbb{D}}$ to the closed regular $p$-gon $\overline{\mathbb{P}}$ homeomorphically. 

\medbreak 

\begin{theorem} \label{text}
The quotient $t = s/c$ extends conformally to $t: \mathbb{P} \to \mathbb{D}$ and homeomorphically to $t: \overline{\mathbb{P}} \to \overline{\mathbb{D}}$ as the inverse of the map 
$$\tau: w \mapsto \int_0^w \frac{{\rm d} \zeta}{(1 + \zeta^p)^{2/p}}\, .$$
\end{theorem} 

\begin{proof} 
It follows from Theorem \ref{t} that in a sufficiently small disc about $0$ and with principal-valued power, 
$$t\,' = (1 + t^p)^{2/p};$$
by integration, it follows that if $z$ is in such a disc then  
$$z = \int_0^{t(z)} \frac{{\rm d} \zeta}{(1 + \zeta^p)^{2/p}}$$ 
with principal-valued power again. An appeal to the Identity Theorem concludes the proof. 

\end{proof} 

\medbreak 

Being centred at $0$ and having the positive real $K_p$ as the midpoint of one of its edges, the regular $p$-gon $\mathbb{P}$ has the complex number $\beta \, L_p$ for one of its vertices, where 
$$L_p = K_p \, \sec (\pi/p)$$
on geometrical grounds. Thus, the regular $p$-gon $\mathbb{P}$ has $K_p$ as the radius of its incircle and $L_p$ as the radius of its circumcircle. The typical vertex of $\mathbb{P}$ is $\beta^n \, L_p$ with $n$ odd: the value of $t$ at this point is $\beta^n$; in particular, $t^p = -1$ at each vertex. The values of $t$ along the edge joining $\beta^{2 m - 1}\,L_p$ to $\beta^{2 m + 1}\,L_p$ run the short arc of the unit circle from $\beta^{2 m - 1}$ to $\beta^{2 m + 1}$;  the value of $t$ at the midpoint $\alpha^m \, K_p$ of this edge is $\alpha^m$.   

\medbreak 

We shall consider the problem of extending $t$ beyond the $p$-gon $\mathbb{P}$ more fully in due course. For now, we merely observe that $t$ does not extend holomorphically in a disc round any vertex of $\mathbb{P}$ unless $p = 3$ or $p = 4$.  

\medbreak 

\begin{theorem} \label{tnot}
If $p > 4$ then $t$ has no holomorphic extension to an open set containing a vertex of $\mathbb{P}$. 
\end{theorem} 

\begin{proof} 
Throughout the proof, we work in a (without loss) connected open set $U$ containing the vertex $a$ of $\mathbb{P}$. Any holomorphic extension of $t$ continues to satisfy 
$$(t\,')^p = (1 + t^p)^2$$
whence we deduce by differentiation that 
$$p (t\,')^{p - 1} t\,'' = 2(1 + t^p) p t^{p - 1} t\,'$$
or 
$$t\,' ((t\,')^{p - 2} t\,'' - 2(1 + t^p) t^{p - 1}) = 0.$$
As $U$ is connected and $t\,' \nequiv 0$ we deduce that 
$$(t\,')^{p - 2} t\,'' = 2(1 + t^p) t^{p - 1}$$
and therefore 
$$(p - 2) (t\,')^{p - 3} (t\,'')^2 + (t\,')^{p - 2} t\,''' = 2 p \,t^{2 p - 2} t\,' + 2(1 + t^p) (p - 1) t^{p - 2} t\,'$$
so that, again dropping $t'$ as a factor,  
$$(p - 2) (t\,')^{p - 4} (t\,'')^2 + (t\,')^{p - 3} t\,''' = 2 p \,t^{2 p - 2} + 2(1 + t^p) (p - 1) t^{p - 2}\,.$$
However, evaluate this purported equality at $a$: the left side is zero because $t\,'(a) = 0$ and $p > 4$; the right side is nonzero because $1 + t(a)^p = 0$ so $t(a) \neq 0$. This contradiction precludes the holomorphic extension of $t$. 
\end{proof} 

\medbreak 

The case $p = 3$ is exceptional: in this Dixonian case, the function $t$ satisfies the perhaps surprising relation $t(z) = - s(- z)$ and is therefore elliptic; see [2] Section 17 and [6] Theorem 4. The case $p = 4$ is also exceptional, $t$ again being elliptic: in fact, $t = - 2 \wp/\wp'$ where $\wp$ is the lemniscatic Weierstrass function with $g_2 = 1$ and $g_3 = 0$; see [7] Theorem 7 and thereafter. 

\medbreak 

Now, we can use the holomorphic extension $t: \mathbb{P} \to \mathbb{D}$ to fashion a solution pair to {\bf IVP}$\C$ in the open regular $p$-gon $\mathbb{P}$; of course, this pair will extend the pair from Theorem \ref{r} on account of the uniqueness clause therein, so we feel free to denote it by the same symbol $(s, c)$.  

\medbreak 

\begin{theorem} \label{scP} 
The system {\bf IVP}$\C$ has a unique holomorphic solution pair $(s, c)$ in the open regular $p$-gon $\mathbb{P}$. 
\end{theorem} 

\begin{proof} 
As $t : \mathbb{P} \to \mathbb{D}$ is holomorphic, the function $1 + t^p : \mathbb{P} \to \C$ is holomorphic and zero-free, whence $T = 1/(1 + t^p) : \mathbb{P} \to \C$ is holomorphic and zero-free. As $\mathbb{P}$ is simply-connected, $T$ has holomorphic $p$th roots in $\mathbb{P}$; let $T^{1/p}$ be the holomorphic $p$th root of $T$ that has value $1$ at $0$. 
 
Now, we define $c: \mathbb{P} \to \C$ and $s: \mathbb{P} \to \C$ by 
$$c = T^{1/p} = (1 + t^p)^{-1/p}$$ 
and 
$$s = t \, T^{1/p} = t \, (1 + t^p)^{-1/p}.$$ 
Plainly, $s(0) = 0$ and $c(0) = 1$. Further, as $t\,' = (1 + t^p)^{2/p} = T^{-2/p}$ so that  
$$T\,' = - (1 + t^p)^{-2} p \, t^{p - 1} t\,' = - p \, t^{p - 1} T^{2 - 2/p}$$
it follows that 
$$c\,' = (1/p) T^{-1 + 1/p} T\,' = - t^{p - 1} T^{1 - 1/p} = - \big(t \, T^{1/p}\big)^{p - 1} = - s^{p - 1}$$
and 
$$s\,' = (T^{-2/p} )T^{1/p} + t (- t^{p - 1} T^{1 - 1/p}) = T^{1 - 1/p} (T^{-1} - t^p) = (T^{1/p})^{p - 1} = c^{p - 1} \,. $$ 
\end{proof} 

\medbreak 

Thus, not only do simultaneous extensions of $s$ and $c$ engender an extension of their quotient: an extension of $s/c$ can generate simultaneous extensions of $s$ and $c$; note here the r\^ole played by simple connectivity. Regarding the proof of Theorem \ref{scP}, the identity $(s/c)\,' = 1/c^2$ (see the proof of Theorem \ref{t}) of course suggests an alternative definition of $c$ as the holomorphic square-root of $1/t\,'$ with value $1$ at $0$. 

\medbreak 

We may double the domain of $t$, $s$ and $c$ as follows. 

\medbreak 

The regular $p$-gon $\mathbb{P}$ has the open segment $(L_p \, \overline{\beta}, L_p \, \beta)$ through $K_p$ as its right edge. Reflexion of $\mathbb{P}$ across this edge produces a congruent regular $p$-gon $\mathbb{P}^+$ with $(L_p \, \overline{\beta}, L_p \, \beta)$ as its left edge and $2 K_p$ as its centre. We shall denote by ${\bf P}$ the union of $\mathbb{P}$ and $\mathbb{P}^+$ along with the open segment $(L_p \, \overline{\beta}, L_p \, \beta)$ that lies between them. 

\medbreak 

Recall the extension $t$ of $s/c$ from Theorem \ref{text}. 

\medbreak 

\begin{theorem} \label{PPt}
The holomorphic function $t : \mathbb{P} \to \mathbb{D}$ extends to a meromorphic function in ${\bf P}$ with a simple pole at $2 K_p$ as its only singularity. 
\end{theorem} 

\begin{proof} 
The function $t : \mathbb{P} \to \mathbb{D}$ extends continuously to the closed $p$-gon $\overline{\mathbb{P}}$ with values in the unit circle around the boundary. The Schwarz Symmetry Principle therefore extends $t$ from $\mathbb{P}$ to ${\bf P}$ by reflexion across $(L_p \, \overline{\beta}, L_p \, \beta)$: if $z^+ = 2 K_p - \overline{z}\in \mathbb{P}^+$ is obtained from $z \in \mathbb{P}$ by reflexion, then $t(z^+) = 1/\overline{t(z)}$. The only (simple) zero of the original $t$ at $0$ corresponds to the only (simple) pole of the extended $t$ at its image $2 K_p$. 
\end{proof} 

\medbreak 

Of course, the meromorphic function $t$ in ${\bf P}$ has $0$ for its only zero. 

\medbreak 
 
Note that the doubled $p$-gon is invariant under reflexion $z \mapsto 2 K_p - z$ through $K_p$; with this and the fact that $t$ is `real' in mind, the meromorphc function $t$ in ${\bf P}$ satisfies 
$$t(2 K_p - z) = 1/t(z).$$ 

\medbreak 

So much for $t$; now for $s$ and $c$. 

\medbreak 

\begin{theorem} \label{PPsc}
The system {\bf IVP}$\C$ has a unique (holomorphic) solution pair $(s, c)$ in the doubled open regular $p$-gon ${\bf P}$.
\end{theorem} 

\begin{proof} 
Upgrade the proof of Theorem \ref{scP} in light of Theorem \ref{PPt}. The meromorphic function $t$ satisfies $t^p = - 1$ nowhere in $\mathbb{P}$ and hence (by Schwarz Symmetry) nowhere in $\mathbb{P}^+$; it also satisfies $t^p = -1$ nowhere on the open segment that lies between these $p$-gons. It follows that the meromorphic function $1 + t^p$ is zero-free in ${\bf P}$ with a pole of order $p$ at $2 K_p$ as its only singularity, whence $T = 1/(1 + t^p)$ is holomorphic in ${\bf P}$ with a (removable) zero of order $p$ at $2 K_p$ as its only zero. As ${\bf P}$ is simply-connected, $T$ has a unique holomorphic $p$th root in ${\bf P}$ with value $1$ at $0$. Take $c$ to be this holomorphic $p$th root and take $s = t \, c$; then proceed as in the proof of Theorem \ref{scP}. 
\end{proof} 

\medbreak 

The proof shows that $c$ has a simple zero at $2 K_p$ and no other zero, while $s$ has a simple zero at $0$ and no other zero. 

\medbreak 

We now proceed to a couple of issues in which the parity of $p$ plays a r\^ole. 

\medbreak 

{\it Let the integer $p > 2$ be even}. 

\medbreak 

We should perhaps begin by noting that $s$ and $c$ now have definite parity. 

\medbreak 

\begin{theorem} \label{parity} 
If $z \in B_r(0)$ then $s$ is odd and $c$ is even. 
\end{theorem} 

\begin{proof} 
We may reuse the device from the proof of Theorem \ref{conj}, showing that $S$ and $C$ defined now by $S(z) = - s(- z)$ and $C(z) = c(-z)$ satisfy {\bf IVP}$\C$ whence $(S, C) = (s, c)$; alternatively, we may apply Theorem \ref{alpha} a total of $\tfrac{1}{2} p$ times. 
\end{proof} 

\medbreak 

Of course, $s$ and $c$ continue to be odd and even when they are extended to connected domains that are invariant under multiplication by $-1$. 

\medbreak 

As $p$ is even, the reflexion $\mathbb{P}^+$ of $\mathbb{P}$ can equally be defined as the translate $\mathbb{P} + 2 K_p$. Let $\mathcal{P}$ be the union of the translates $\{ \mathbb{P} + 2 n K_p : n \in \mathbb{Z} \}$ together with the open segments that lie between adjacent translates; thus, $\mathcal{P}$ is an open neighbourhood of the real axis, centred about which it includes an open band of vertical half-width $L_p \, \sin (\pi/p) = K_p \, \tan (\pi/p)$. 

\medbreak 

\begin{theorem} \label{calt}
The quotient $t = s/c$ extends to the open polygonal band $\mathcal{P}$ as a meromorphic function of period $4 K_p$. 
\end{theorem} 

\begin{proof} 
Recall Theorem \ref{PPt}: the extension $t$ to the doubled $p$-gon ${\bf P}$ therein further extends continuously to the boundary with values in the unit circle; by construction, the values of $t$ running up the right edge of ${\bf P}$ copy the values of $t$ running up the left edge of ${\bf P}$. By the Schwarz Symmetry Principle, repeated reflexions extend $t$ to the whole band $\mathcal{P}$ in a manner that is evidently periodic with period equal to the width of ${\bf P}$. 
\end{proof} 

\medbreak 

We remark that the extended $t$ has simple zeros and simple poles, its full zero-set being $\{ 4 n K_p : n \in \mathbb{Z} \}$ and its full pole-set being $\{ (4 n + 2) K_p : n \in \mathbb{Z} \}$. Also, $t$ takes values of unit modulus precisely on the open segments between adjacent $p$-gons; beyond this, $t$ approaches unit modulus at the boundary points of $\mathcal{P}$. 

\medbreak 

The functions $s$ and $c$ also extend to the same band. 

\medbreak 

\begin{theorem} \label{calsc}
The system {\bf IVP}$\C$ has a unique holomorphic solution pair $(s, c)$ in the open polygonal band $\mathcal{P}$. 
\end{theorem} 

\begin{proof} 
A further upgrade to the proof of Theorem \ref{scP} along the lines of Theorem \ref{PPsc}; we need only observe that $\mathcal{P}$ is simply-connected. 
\end{proof} 

\medbreak 

The following complementarity law has a familiar counterpart in the root case $p = 2$. 

\medbreak 

\begin{theorem} \label{trig}
If $z \in \mathcal{P}$ then $c(2 K_p - z) = s(z)$ and $s(2 K_p - z) = c(z)$. 
\end{theorem} 

\begin{proof} 
The identity $t(2 K_p - z) = 1/t(z)$ (noted after Theorem \ref{PPt}) holding in ${\bf P}$ continues to hold in $\mathcal{P}$. It follows that $T = 1/(1 + t^p)$ satisfies 
$$T(2 K_p - z) = t(z)^p T(z)$$
whence passage to the holomorphic $p$th root yields 
$$c(2 K_p - z) = t(z) c(z)$$
since $t(K_p) = 1$ and $c(K_p) \neq 0$; this proves that 
$$c(2 K_p - z) = s(z).$$
The companion identity 
$$s(2 K_p - z) = c(z)$$
follows either upon the replacement of $z$ by $2 K_p - z$ or upon calculating 
$$s(2 K_p - z) = c(2 K_p - z) t(2 K_p - z) = s(z)/t(z) = c(z).$$ 
\end{proof} 

\medbreak 

The functions $s: \mathcal{P} \to \C$ and $c: \mathcal{P} \to \C$ are periodic. 

\medbreak 

\begin{theorem} \label{per}
The solution pair $(s, c)$ to {\bf IVP}$\C$ in $\mathcal{P}$ has period $8 K_p$. 
\end{theorem} 

\begin{proof} 
From Theorem \ref{parity} and Theorem \ref{trig} we deduce that 
$$c(2 K_p + z) = s(- z) = - s(z)$$ 
and 
$$s(2 K_p + z) = c(- z) = c(z).$$ 
Repeat: 
$$c(4 K_p + z) = - s(2 K_p + z) = - c(z)$$ 
and 
$$s(4 K_p + z) = c(2 K_p + z) = - s(z).$$ 
Repeat. 
\end{proof} 

\medbreak 

Note that the zero-set of $s$ is $\{ 4 n K_p : n \in \mathbb{Z} \}$ and the zero-set of $c$ is $\{ (4 n + 2) K_p : n \in \mathbb{Z} \}$. Of course, $s$ continues to be odd and $c$ continues to be even, as in Theorem \ref{parity}; and both functions continue to be `real' in the sense of Theorem \ref{conj}.   

\medbreak 

{\it Let the integer $p > 2$ be odd}. 

\medbreak 

 In place of the definite parity displayed in Theorem \ref{parity}, $s$ and $c$ now have the following symmetry properties. 

\medbreak 

\begin{theorem} \label{beta} 
If $z \in B_r (0)$ then $s(- \beta z) = - \beta s(z)$ and $c(- \beta z) = c(z)$. 
\end{theorem} 

\begin{proof} 
We may again reuse the familiar device from the proof of Theorem \ref{conj}, this time with $S(z) = - \overline{\beta} s(- \beta z)$ and $C(z) = c(- \beta z)$. The oddness of $p$ is needed to secure the correct sign in $C\,' = - S^{p - 1}$. 
\end{proof} 

As a consequence, the multiplicative action of $\beta$ has the following effect: 
$$s(\beta z) = s( - \beta (- z)) = - \beta s(- z)$$
$$c(\beta z) = c( - \beta (- z)) = c(- z).$$
Of course, these properties continue to be satisfied when $s$ and $c$ are extended to an appropriately symmetric connected domain, such as the regular $2 p$-gon $\mathbb{P} \cap \beta \mathbb{P}$. 

\medbreak 

As $p$ is odd, the doubled $p$-gon ${\bf P}$ of Theorem \ref{PPt} no longer has a `left edge' and a `right edge': instead, it has a leftmost vertex $P^-$ at $- L_p$ and a rightmost vertex $P^+$ at $2 K_p + L_p$; the centre $P$ of ${\bf P}$ is at $K_p$. Now, let us further consider Theorem \ref{PPt} and Theorem \ref{PPsc} in this context. The value of $t$ at the centre $P$ is $t(K_p) = 1$; the functions $s$ and $c$ share the same value there, namely $2^{-1/p}$ on account of Theorem \ref{Pyth}. The value of $t$ is $- 1$ at the ends $P^+$ and $P^-$ but the functions $s$ and $c$ (which have real output for real input according to Theorem \ref{conj}) do not have `equal-but-opposite' values at these ends: the equality $c = - s$ holding there would force the contradiction 
$$1 = s^p + c^p = s^p + (- s)^p = s^p - s^p = 0$$
since $p$ is odd. Instead, as $x \in (-L_p, 2 K_p + L_p)$ approaches $P^{\pm}$ it follows that $s(x) \to \pm \infty$ and $c(x) \to \mp \infty$ and therefore that $t(x) = s(x)/c(x) \to -1$ because 
$$\big(s(x)/c(x)\big)^p = (1 - c(x)^p)/c(x)^p  = 1/c(x)^p - 1 \to - 1.$$ 
We shall have more to say on some of these points in the Remarks section. 

\medbreak 

\section*{Remarks}

\medbreak 

We close our account with some miscellaneous remarks, generally leaving their full proofs as exercises.  

\medbreak 

It is perhaps needless to point out that there are numerous alternative routes through the material contained in this paper. We make no claim to have been especially expeditious in our progress. 

\medbreak 

Much of the development in the section `Real Systems' is superseded by that in `Complex Systems'. For example, the solution pair in Theorem \ref{calsc} specializes to the solution pair in Theorem \ref{realline} upon restriction to $\R$. In the opposite direction, since {\bf IVP}$\C$ is autonomous, an application of the Picard existence-uniqueness theorem shows that the solution pair in Theorem \ref{realline} has a complex extension at least to the open band of half-width $(p - 2)^{p - 2} / (p - 1)^{p - 1}$ centred on $\R$. Incidentally, it is of course possible to carry the solution pair $(s_0, c_0)$ of Theorem \ref{-AA} forward beyond the interval $[- A_p, A_p]$ by continuing the differential equations $s\,' = c^{p - 1}$ and $c\,' = - s^{p - 1}$ but from the new initial conditions $s(A_p) = 1$ and $c(A_p) = 0$. 

\medbreak 

A variant of the approach taken in `Real Systems' can be pursued in `Complex Systems'. Thus, integration of the initial value problem 
$$s\,' = (1 - s^p)^{1 - 1/p}; \, s(0) = 0$$ 
reveals $s$ as inverse to the function $\sigma$ defined by 
$$\sigma (w) = \int_0^w \frac{{\rm d} \zeta}{(1 - \zeta^p)^{1 - 1/p}}\, .$$
This makes contact with hypergeometric function theory: thus, 
$$\sigma(w) = w \, F(1/p, 1 - 1/p; 1 + 1/p; w^p)$$ 
where $F$ stands for $_2 F_1$ as usual. In the root case $p = 2$ this is the familiar formula 
$$\arcsin (w) = w \, F(1/2, 1/2; 3/2; w^2);$$
in the Dixonian case $p = 3$ it is the formula 
$${\rm arcsm}(w) = w \, F(1/3, 2/3; 4/3; w^3)$$
essentially as recorded in Proposition 1 of [1]. The approach that we choose to follow  in `Complex Systems' passes via the regular $p$-gon from the ratio $t = s/c$: this ratio is inverse to the function $\tau$ defined by 
$$\tau (w) = \int_0^w \frac{{\rm d} \zeta}{(1 + \zeta^p)^{2/p}} = w \, F(1/p, 2/p; 1 + 1/p; - w^p)\,.$$
In the root case $p = 2$ this is the familiar 
$$\arctan (w) = w \, F(1/2, 1; 3/2; - w^2)\,.$$ 
In the Dixonian case, the identity $t(z) = - s(- z)$ recalled after Theorem \ref{tnot} explains why the hypergeometric formula for $\tau$ so closely resembles that for $\sigma$. 

\medbreak 

Our account has turned up two sets of positive numbers: $A_p$ and $B_p$ in `Real Systems'; $K_p$ and $L_p$ in `Complex Systems'; a little thought exposes how these are related. In `Real Systems', $A_p$ is the least positive number at which $s$ takes the value $1$ and $c$ takes the value $0$; in `Complex Systems', $t = s/c$ is positive on $(0, 2 K_p)$ and has a pole at $2 K_p$. Thus 
$$A_p = 2 K_p\,.$$
Incidentally, this can be read as an equality between integrals: on the left is the time taken for $s$ to reach value $1$; on the right is the time taken for $t$ to become infinite. The substitution $\xi = (1 + \eta^p)^{-1/p}$ provides a direct justification of this equality in the integral form 
$$\int_0^1 \frac{{\rm d} \xi}{(1 - \xi^p)^{1 - 1/p}} =  \int_0^{\infty} \frac{{\rm d} \eta}{(1 + \eta^p)^{2/p}}\,.$$
 In a similar vein, the fact that the time $K_p$ taken for $t$ to reach value $1$ is half the time $2 K_p$ taken for $t$ to become infinite (familiar when $p = 2$) is reflected in the independently verifiable integral identity 
$$ \int_0^1 \frac{{\rm d} \eta}{(1 + \eta^p)^{2/p}} = \tfrac{1}{2}  \int_0^{\infty} \frac{{\rm d} \eta}{(1 + \eta^p)^{2/p}}\,.$$ 
As noted in connexion with Theorem \ref{text},  $t(\beta \, L_p) = \beta$: thus
$$\beta \, L_p = \int_0^{\beta} \frac{{\rm d} \zeta}{(1 + \zeta^p)^{2/p}} = \int_0^1 \frac{\beta \, {\rm d} u}{(1 + \beta^p u^p)^{2/p}} = \beta \int_0^1 \frac{{\rm d} u}{(1 - u^p)^{2/p}}$$
and the substitution $u = (1 + v^{-p})^{-1/p}$ justifies 
$$\int_0^1 \frac{{\rm d} u}{(1 - u^p)^{2/p}} = \int_0^{\infty} \frac{{\rm d} v}{(1 + v^p)^{1 - 1/p}}$$
so that 
$$L_p = B_p\, .$$
Of course, the identities $A_p = 2 K_p$ and $L_p = B_p$ just established are consistent with the earlier identities $A_p = 2 B_p \cos(\pi/p)$ (after Theorem \ref{odd}) and $K_p = L_p \cos(\pi/p)$ (after Theorem \ref{text}). 

\medbreak 

In particular, note that the interval $(- B_p, A_p + B_p)$ that was encountered near the close of `Real Systems' coincides with the interval $(-L_p, 2 K_p + L_p) = \R \cap {\bf P}$ that was encountered at the close of `Complex Systems'. 

\medbreak 

\bigbreak

\begin{center} 
{\small R}{\footnotesize EFERENCES}
\end{center} 
\medbreak 

[1] E. Conrad and P. Flajolet, {\it The Fermat cubic, elliptic functions, continued fractions, and a combinatorial excursion}, S\'eminaire Lotharingien de Combinatoire 54 (2006) Article B54g.

\medbreak 

[2]  A.C. Dixon, {\it On the doubly periodic functions arising out of the curve $x^3 + y^3 - 3 \alpha x y = 1$}, The Quarterly Journal of Pure and Applied Mathematics, 24 (1890) 167-233. 

\medbreak 

[3] E. Hille, {\it Ordinary Differential Equations in the Complex Domain}, Wiley-Interscience (1976); Dover Publications (1997).

\medbreak 

[4] J.C. Langer and D.A. Singer, {\it The Trefoil}, Milan Journal of Mathematics, 82 (2014) 161-182. 

\medbreak 

[5] A. Levin, {\it A Geometric Interpretation of an Infinite Product for the Lemniscate Constant}, The American Mathematical Monthly, Volume 113 Number 6 (2006) 510-520.

\medbreak 

[6] P.L. Robinson, {\it The Dixonian Elliptic Functions}, arXiv 1901.04296 (2019). 

\medbreak 

[7] P.L. Robinson, {\it The elliptic functions in a first-order system}, arXiv 1903.07147 (2019). 

\medbreak 

[8] P.L. Robinson, {\it Meromorphic functions in a first-order system}, arXiv 1906.02141 (2019). 

\medbreak

\end{document}